\documentclass[12pt,reqno]{amsart}
\usepackage{amsmath,amscd,amsthm,a4wide,amssymb,amsxtra,mathrsfs,amsrefs}

\allowdisplaybreaks

\theoremstyle{plain}
\newtheorem{thm}{Theorem}[section]
\newtheorem{lem}[thm]{Lemma}
\newtheorem{cor}[thm]{Corollary}

\theoremstyle{definition}
\newtheorem{rem}[thm]{Remark}

\newtheorem{ex}[thm]{Example}

\numberwithin{equation}{section}

\newtoks\by
\newtoks\paper
\newtoks\book
\newtoks\jour
\newtoks\yr
\newtoks\pages
\newtoks\vol
\newtoks\publ
\def\ota{{\hbox\vol{???}}}
\def\cLear{\by=\ota\paper=\ota\book=\ota\jour=\ota\yr=\ota
\pages=\ota\vol=\ota\publ=\ota}
\def\endpaper{\the\by, \textit{\the\paper},
{\the\jour} \textbf{\the\vol} (\the\yr), \the\pages.\cLear}
\def\endbook{\the\by, \textit{\the\book}, \the\publ.\cLear}
\def\endprep{\the\by, \textit{\the\paper}, \the\jour.\cLear}

\def\loc{\operatorname{loc}}

\def\R{\mathbb R}

\def\S{\mathbb S}

\def\rn{\R^n}
\def\a{\alpha}

\def\la{\lambda}

\def\vp{\varphi}

\def\i{\infty}

\def\d{\delta}

\def\ls{\lesssim}

\def\R{\mathbb R}

\def\d{\delta}
\def\e{\varepsilon}

\def\Lloc{L_1^{\rm loc}(\rn)}

\def\B{\operatorname{BMO}}

\begin{document}

\title{A note on maximal commutators and commutators of maximal
functions}
\author{Mujdat \textsc{Agcayazi}, Amiran \textsc{Gogatishvili}, Kerim \textsc{Koca}, Rza \textsc{Mustafayev}}
\thanks{The research  of A. Gogatishvili was partialy
supported by the grant 201/08/0383 and 13-14743S of the Grant Agency of the
Czech Republic and RVO: 67985840. The research of R.Ch. Mustafayev
was supported by the Science Development Foundation under the
President of the Republic of Azerbaijan Project No.
EIF-2010-1(1)-40/06-1}




\begin{abstract}
In this paper maximal commutators and commutators of maximal
functions with functions of bounded mean oscillation are
investigated. New pointwise estimates for these operators are
proved.
\end{abstract}
\maketitle
\section{Introduction }

Given a locally integrable function $f$ on $\rn$, the
Hardy-Littlewood maximal function $Mf$ of $f$ is defined by
$$
Mf(x):=\sup_{Q \ni x} \frac{1}{|Q|}\int_Q |f(y)|\,dy, \qquad
(x\in\rn),
$$
where the supremum is taken over all cubes $Q$ containing $x$. The
operator $M:~f \rightarrow Mf$ is called the Hardy-Littlewood
maximal operator.

For any $f\in\Lloc$ and $x \in \rn$, let $M^{\#} f$ be the sharp
maximal function of Fefferman-Stein defined by
\begin{equation*}
M^{\#}f(x) = \sup_{Q\ni x}\frac{1}{|Q|}\int_Q |f(y)-f_Q|dy,
\end{equation*}
where the supremum extends over all cubes containing $x$, and
$f_Q$ is the mean value of $f$ on $Q$. For a fixed $\d \in (0,1)$,
any suitable function $g$ and $x\in\rn$, let $M_{\d}^{\#} g (x) =
[M^{\#} (|g|^{\d})(x)]^{1/{\d}}$ and $M_{\d}g (x) = [M
(|g|^{\d})(x)]^{1/{\d}}$.

Let $f\in\Lloc$. Then $f$ is said to have bounded mean oscillation
($f\in \B$) if the seminorm given by
\begin{equation*}
\|f\|_{*}:= \sup_Q\frac{1}{|Q|}\int_Q |f(y)-f_Q|dy
\end{equation*}
is finite.

Let $T$ be the Calder\'on-Zygmund singular integral operator
$$
Tf(x) : = \mbox{p.v.} \int_{\rn} K(x-y)f(y)\,dy
$$
with kernel $K(x) = \Omega (x) / |x|^n$, where $\Omega$ is
homogeneous of degree zero, infinitely differentiable on the unit
sphere $\S^{n-1}$, and $\int_{\S^{n-1}} \Omega = 0$.

The well-known result of Coifman, Rochberg, and Weiss \cite{CRW}
states that if $b \in \B(\rn)$, then $[T,b]$ defined initially for
$f\in L_c^{\i}(\rn)$, by
\begin{equation}\label{CRWoperator}
[T,b](f) = T(bf) - bT(f),
\end{equation}
is bounded on $L^p(\rn)$, $1 < p < \i$; conversely, if $[R_i,b]$ is
bounded on $L^p(\rn)$ for every Riesz transform $R_i$, then $b \in
\B(\rn)$. Janson \cite{j} observed that actually for any singular
integral $T$ (with kernel satisfying the above-mentioned conditions)
the boundedness of $[T,b]$ on $L^p(\rn)$ implies $b \in \B(\rn)$.


Unlike the classical theory of singular integral operators, simple
examples show that $[T,b]$ fails to be of weak-type (1,1) when $b
\in \B$, and satisfies weak-type $L(\log L)$ inequality (see
\cite{CPer}).

It is well-known that the maximal commutator
$$
C_b(f)(x):=\sup_{x\in Q}\frac{1}{|Q|}\int_{Q} |b(x)-b(y)||f(y)|dy,
\qquad (x\in\rn),
$$
plays an important role in the study of commutators of singular
integral operators with $\B$ symbols (see, for instance,
\cite{GHST}, \cite{LiHuShi}, \cite{ST1}, \cite{ST2}).
Garcia-Cuerva et al. \cite{GHST} proved that $C_b$ is bounded in
$L_p(\rn)$ for any $p\in (1,\i)$ if and only if $b\in\B(\rn)$, and
Alphonse \cite{Alphonse} proved that $C_b$ enjoys weak-type
$L(\log L)$ estimate, that is, inequality
\begin{equation}\label{eq0003.7}
|\{x \in\rn :  C_b(f)(x)> \la \}| \le C
\int_{\rn}\frac{|f(x)|}{\la}\left(1+\log^+
\left(\frac{|f(x)|}{\la}\right)\right)dx
\end{equation}
holds. The maximal operator $C_b$ was studied intensively and
there exist plenty of results about it.

The commutator of the Hardy-Littlewood maximal operator $M$ and a
locally integrable function $b$ is defined by
\begin{equation*}
[M,b]f = M(bf)-bMf.
\end{equation*}
The operator $[M,b]$ was studied by Milman et al. in
\cite{MilSchon} and \cite{BasMilRu}. This operator arise, for
example, when trying to give a meaning to the product of a
function in $H^1$ (which may not be a locally integrable function)
and a function in $\B$ (see, for instance, \cite{bijz}). Using
real interpolation techniques, in \cite{MilSchon}, Milman and
Schonbek proved the $L_p$-boundedness of the operator $[M,b]$. In
\cite{BasMilRu}, Bastero, Milman and Ruiz proved that the
commutator of the maximal operator with locally integrable
function $b$ is bounded in $L_p$ if and only if $b$ is in $\B$
with bounded negative part. As we know only these two papers are
devoted to the problem of boundedness of the commutator of maximal
function in Lebesgue spaces.

In this paper the maximal commutator and the commutator of the
maximal function with functions of bounded mean oscillation are
investigated. Although  operators $C_b$ and $[M,b]$ essentially
differ from each other (for example, $C_b$ is a positive and a
sublinear operator, but $[M,b]$ is neither a positive nor a
sublinear), but if $b$ satisfies some additional conditions, then
operator $C_b$ controls $[M,b]$. If $b$ is a non-negative locally
integrable function, then the following inequality
\begin{equation*}
|[M,b]f|(x)\leq C_b(f)(x) \qquad (x\in\rn)
\end{equation*}
holds. If $b$ is any locally integrable function on $\rn$, then the
estimation
\begin{equation}\label{eq.0001}
|[M,b]f|(x)\leq C_b(f)(x)+ 2b^-(x) Mf(x) \qquad (x\in\rn)
\end{equation}
is true.\footnote{Denote by $b^+(x)=\max\{b(x),0\}$ and
$b^-(x)=-\min\{b(x),0\}$, consequently $b=b^+-b^-$ and
$|b|=b^++b^-$.}

Now we state our main results.
\begin{thm}\label{main}
Let $b\in \B$ and let $0<\d  <1$. Then, there exists a positive
constant $C=C_{\d}$ such that
\begin{equation}\label{eqmain2}
M_{\d}(C_b(f))(x)\leq C \|b\|_{*} M^2f(x) \qquad (x\in\rn)
\end{equation}
for all functions from $L_1^{\loc}(\rn)$.
\end{thm}

Inequality \eqref{eqmain2} improves known inequality
\begin{equation*}
M_{\d}^{\#} (C_b(f))(x)\ls \|b\|_{*}M^2f(x),
\end{equation*}
which involves the Fefferman-Stein sharp maximal function (see
\cite{HuYang}, Lemma 1, for instance). Indeed, since $M_{\d}^{\#}
\ls M_{\d}$, then
$$
M_{\d}^{\#} (C_b(f))(x)\lesssim M_{\d} (C_b(f))(x)\leq
C\|b\|_{*}M^2f(x) \qquad (x\in\rn).
$$

Since, by the Lebesgue differentiation theorem
$$
C_b(f)(x) \leq M_{\d}(C_b(f))(x),
$$
the following statement follows from Theorem \ref{main}.
\begin{thm}\label{lem1111111}
Let $b\in \B$. Then, there exists a positive constant $C$ such that
\begin{equation}\label{eq.0002}
C_b(f)(x)\leq C \|b\|_{*} M^2f(x) \qquad (x\in\rn)
\end{equation}
for all functions from $L_1^{\loc}(\rn)$.
\end{thm}
Using the boundedness of maximal operator $M$ in $L_p(\rn)$, $1 <
p < \i$, in view of inequality \eqref{eq.0002}, we obtain new
proof of sufficient part of above-mentioned result by
Garcia-Cuerva et al. from \cite{GHST}. Moreover, inequality
\eqref{eq.0002} allows us to obtain another proof of endpoint
estimate for the operator $C_b$ from \cite{Alphonse}. We
strengthen this theorem by showing that weak endpoint estimate for
$C_b$ implies $b\in\B$ (Note that in \cite{Alphonse} only
sufficient part of Theorem \ref{mainthm} was proved). On the other
hand, inequalities \eqref{eq.0001} and \eqref{eq.0002} give new
more easy (in our opinion) proof of aforesaid results by Milman et
al. in \cite{MilSchon} and \cite{BasMilRu}.

On using inequalities \eqref{eq.0001} and \eqref{eq.0002}, we
consider weak-type estimate for the commutator of the maximal
function as well, and prove the estimate
\begin{equation*}
|\{x \in\rn : |[M,b]f(x)| > \la \}| \\
\ls\int_{\rn}\frac{|f(x)|}{\la}\left(1+\log^+
\left(\frac{|f(x)|}{\la}\right)\right)dx,
\end{equation*}
which is new, as we know.

From \eqref{eq.0001} and \eqref{eq.0002} we easily get the
following statement.
\begin{thm}\label{lem0003.3}
Let $b$ is in $\B(\rn)$ such that $b^-\in L_{\i}(\rn)$. Then, there
exists a positive constant $C$ such that
\begin{equation}\label{eqPointwise3}
|[M,b]f(x)|\leq C \left(\|b^+\|_{*}+\|b^-\|_{\i}\right)M^2f(x)
\end{equation}
for all functions from $L_1^{\loc}(\rn)$.
\end{thm}
Inequalities \eqref{eq.0002} and \eqref{eqPointwise3} allows us to
state the boundedness of both operators in any Banach space of
measurable functions defined on $\rn$ which the Hardy-Littlewood
maximal operator lives invariant.
\begin{thm}
Let $b \in \B(\rn)$. Suppose that $X$ is a Banach space of
measurable functions defined on $\rn$. Assume that $M$ is bounded on
$X$. Then the operator $C_b$ is bounded on $X$, and the inequality
$$
\|C_b f\|_X \le C \|b\|_* \|f\|_X
$$
holds with constant $C$ independent of $f$.

Moreover, if $b^- \in L_{\i}(\rn)$, then the operator $[M,b]$ is
bounded on $X$, and the inequality
$$
\|[M,b] f\|_X \le C (\|b^+\|_* + \|b^-\|_{\i})\|f\|_X
$$
holds with constant $C$ independent of $f$.
\end{thm}

The paper is organized as follows. Section \ref{sect2} contains
some preliminaries along with the standard ingredients used in the
proof. In Section \ref{sect3} we prove Theorem \ref{main}. In
Section \ref{sect4} we apply the main result to get another proof
of endpoint weak-type estimate for operator $C_b$. In Section
\ref{sect5} we obtain pointwise estimate of$[M,b]$ by iterated
maximal function. Using this estimate we get new proof of the
$L_p$ - boundedness. Finally, in Section \ref{sect6} we prove
endpoint estimate for the commutator of the maximal function with
$\B$ functions.

\section{Notations and Preliminaries}\label{sect2}

Now we make some conventions. Throughout the paper, we always
denote by $c$ and $C$ a positive constant, which is independent of
main parameters, but it may vary from line to line. However a
constant with subscript such as $C_1$ does not change in different
occurrences. By $A\ls B$ we mean that $A\le C B$ with some
positive constant $C$ independent of appropriate quantities. If
$A\ls B$ and $B\ls A$, we write $A\approx B$ and say that $A$ and
$B$ are equivalent. For a measurable set $E$, $\chi_E$ denotes the
characteristic function of $E$. Throughout paper cubes will be
assumed to have their sides parallel to the coordinate axes. Given
$\la > 0$ and a cube $Q$, $\la Q$ denotes the cube with the same
center as $Q$ and whose side is $\la$ times that of $Q$. For a
fixed $p$ with $p\in [1,\i)$, $p'$ denotes the dual exponent of
$p$, namely, $p'=p/(p-1)$. For any measurable set $E$ and any
integrable function $f$ on $E$, we denote by $f_E$ the mean value
of $f$ over $E$, that is, $f_E=(1/|E|)\int_E f(x)dx$.

For the sake of completeness we recall the definitions and some
properties of the spaces we are going to use.

The non-increasing rearrangement (see, e.g., \cite[p. 39]{BS}) of a
measurable function $f$ on $\rn$ is defined by
$$
f^*(t)=\inf\left\{\la >0 : |\{x\in\rn: |f(x)|>\la \}| \leq t
\right\}\quad (0<t<\i).
$$
Let $p\in[1,\infty)$. The Lorentz space $L^{p,\infty}$ is defined as
$$
L^{p,\infty}(\rn):= \left\{f:\quad  f \, \text{measurable on}\, \,
\rn; \quad
\|f\|_{L^{p,\infty}(\rn)}:=\sup_{0<t<\infty}t\sp{\frac1{p}}f^*(t)<\infty\right\}.
$$

The most important result regarding $\B$ is the following theorem
of F.~John and L.~Nirenberg \cite{JN} (see, \cite{GR}, p. 164,
also).
\begin{thm}
There exists constants $C_1$, $C_2$ depending only on the dimension
$n$, such that for every $f\in \B=\B(\rn)$, every cube $Q$ and every
$t>0$:
\begin{equation}\label{eq0002.2}
|\{x\in Q: |f(x)-f_Q|>t \}|\leq C_1 |Q| \exp
\left\{-\frac{C_2}{\|f\|_*}t\right\}.
\end{equation}
\end{thm}

\begin{lem}[\cite{GR}, p. 166]\label{lem002.4.}
Let $f\in \B$ and $p\in (0,\i)$. Then for every $\la$ such that
$0<\la <C_2/\|f\|_*$, where $C_2$ is the same constant appearing
in \eqref{eq0002.2}, we have
$$
\sup_{Q} \frac{1}{|Q|}\int_{Q}\exp\{\la|f(x)-f_{Q}|\}dx<\i.
$$
\end{lem}

\begin{lem}[\cite{JN} and \cite{BenSharp}]\label{lem2.4.}
For $p\in (0,\i)$, $\B(p)=\B$, with equivalent norms, where
$$
\|f\|_{\B(p)}:=\sup_{Q}\left( \frac{1}{|Q|}\int_{Q}|f(y)-f_{Q}|^p
dy\right)^{\frac{1}{p}}.
$$
\end{lem}

%
%

A continuously increasing function on $[0,\i]$, say $\Psi:
[0,\i]\rightarrow [0,\i]$ such that $\Psi (0)=0$, $\Psi(1)=1$ and
$\Psi(\i)=\i$, will be referred to as an Orlicz function. If
$\Psi$ is a Orlicz function, then
$$
\Phi (t)=\sup \{ts-\Psi(s): s\in [0,\i]\}
$$
is the complementary Orlicz function to $\Psi$.

Let $\Omega$ be a cube of $\rn$. The generalized Orlicz space
denoted by $L^{\Psi}(\Omega)$ consists of all functions $g: \Omega
\subseteq \rn \rightarrow \R$ such that
$$
\int_{\Omega} \Psi \left(\frac{|g|}{\a}\right)(x)dx <\i
$$
for some $\a >0$.

Let us define the $\Psi$-average of $g$ over a cube $Q$ contained in
$\Omega$ by
\begin{equation*}
\|g\|_{\Psi,Q}=\inf \left\{\la>0: \frac{1}{|Q|}\int_{Q}
\Psi\left(\frac{|g(y)|}{\a}\right)dy\leq 1\right\}.
\end{equation*}
When $\Psi$ is a Young function, i.e. a convex Orlicz function, the
quantity
\begin{equation*}
\|f\|_{\Psi}=\inf \left\{\la>0: \int_{\Omega}
\Psi\left(\frac{|f(y)|}{\a}\right)dy\leq 1\right\}
\end{equation*}
is well-known Luxemburg norm in the space $L^{\Psi}(\Omega)$ (see
\cite{RR}).

If $f\in L^{\Psi}(\rn)$, the maximal function of $f$ with respect to
$\Psi$ is defined by setting
\begin{equation*}
M_{\Psi}f(x)=\sup_{Q \ni x} \|f\|_{\Psi,Q}.
\end{equation*}

The generalized H\"older's inequality
\begin{equation}\label{genHolder}
\frac{1}{|Q|}\int_{Q} |f(y)g(y)|dy \leq \|f\|_{\Phi,Q}
\|g\|_{\Psi,Q},
\end{equation}
where $\Psi$ is the complementary Young function associated to
$\Phi$, holds.

The main example that we are going to be using is
$\Phi(t)=t(1+\log^+ t)$ with maximal function defined by $M_{L(\log
L)}$. The complementary Young function is given by
$\Psi(t)\thickapprox e^t$ with the corresponding maximal function
denoted by $M_{\exp L}$.

Recall the definition of quasinorm of Zygmund space:
$$
\|f\|_{L(1+\log^+ L)}:= \int_{\rn} |f(x)|(1+\log^+ |f(x)|)dx
$$

The size of $M^2$ is given by the following.
\begin{lem}[\cite{CPer}, Lemma 1.6]\label{lem002.8}
There exists a positive constant $C$ such that for any function $f$
and for all $\la >0$,
\begin{equation}
|\{x  \in\rn :  M^2f(x)> \la \}|  \leq C
\int_{\rn}\frac{|f(x)|}{\la}\left(1+\log^+
\left(\frac{|f(x)|}{\la}\right)\right)dx.
\end{equation}
\end{lem}

\section{Proof of main estimate}\label{sect3}

\noindent{\bf \textit {Proof of Theorem \ref{main}.}} Let $x\in \rn$
and fix a cube $x\in Q$. Let $f=f_1+f_2$, where $f_1=f\chi_{3Q}$.
Since for any $y\in \rn$
\begin{equation*}
\begin{split}
C_b(f)(y)&=M((b-b(y))f)(y)=M((b-b_{3Q}+b_{3Q}-b(y))f)(y)\\
& \leq M((b-b_{3Q})f_1)(y)+M((b-b_{3Q})f_2)(y)+|b(y)-b_{3Q}|Mf(y),
\end{split}
\end{equation*}
we have
\begin{align}
\left( \frac{1}{|Q|}\int_{Q} \left(C_b(f)(y)\right)^{\d}
dy\right)^{\frac{1}{\d}}  \lesssim & \left(
\frac{1}{|Q|}\int_{Q}|M((b-b_{3Q})f_1)(y)|^{\d}
dy\right)^{\frac{1}{\d}} \notag\\
& +\left( \frac{1}{|Q|}\int_{Q}|M((b-b_{3Q})f_2)(y)|^{\d}
dy\right)^{\frac{1}{\d}} \notag\\
& +\left( \frac{1}{|Q|}\int_{Q}||b(y)-b_{3Q}|^{\d}(Mf(y))^{\d}
dy\right)^{\frac{1}{\d}}\notag\\
= & \,{\rm I}+{\rm II}+{\rm III}. \label{eq003.4}
\end{align}
Since
\begin{align*}
\int_{Q}|M((b-b_{3Q})f_1)(y)|^{\d}
dy & =\int_0^{|Q|}\left[{\left(M((b-b_{3Q})f_1)\right)^*(t)}\right]^{\d}dt\\
& \leq
\left[\sup_{0<t<|Q|}t\left(M((b-b_{3Q})f_1)\right)^*(t)\right]^{\d}\int_0^{|Q|}t^{-\d}dt,
\end{align*}
using the boundedness of $M$ from $L_1(\rn)$ to $L^{1,\i}(\rn)$ we
have
\begin{align*}
\int_{Q}|M((b-b_{3Q})f_1)(y)|^{\d} dy & \lesssim
\|(b-b_{3Q})f_1\|_{L_1(\rn)}^{\d}
|Q|^{-\d+1}\\
&=\|(b-b_{3Q})f\|_{L_1(3Q)}^{\d} |Q|^{-\d+1}.
\end{align*}
Thus
\begin{equation*}
{\rm I} \lesssim \frac{1}{|Q|}\int_{3Q}|b(y)-b_{3Q}||f(y)|dy.
\end{equation*}
By generalized H\"older's inequality \eqref{genHolder}, we get
\begin{equation*}
{\rm I} \lesssim \|b-b_{3Q}\|_{\exp L,3Q} \|f\|_{L \log L,3Q}.
\end{equation*}
Since by Lemma \ref{lem002.4.}, there is a  constant $C>0$ such that
for any cube $Q$,
$$
\|b-b_{Q}\|_{\exp L,Q} \leq C \|b\|_*,
$$
we arrive at
\begin{equation}\label{eqI}
{\rm I} \lesssim \|b\|_{*} M_{L \log L}f(x).
\end{equation}

Let us estimate ${\rm II}$. Since ${\rm II}$ is comparable to
$\inf_{y\in Q}M((b-b_{3Q})f)(y)$ (see \cite{GR},~p. 160, for
instance), then
$$
{\rm II}\lesssim M((b-b_{3Q})f)(x).
$$
Again by generalized H\"older's inequality and Lemma
\ref{lem002.4.}, we get
\begin{equation}\label{eqII}
{\rm II} \lesssim  \sup_{x\in Q}\|b-b_{3Q}\|_{\exp L, 3Q} \|f\|_{L
\log L,3Q} \lesssim \|b\|_{*} M_{L \log L}f(x).
\end{equation}

Let $\d<\e<1$. To estimate ${\rm III}$ we use H\"older's inequality
with exponents $r$ and $r'$, where $r=\e /\d >1$:
\begin{equation*}
\begin{split}
{\rm III} \leq \left( \frac{1}{|Q|}\int_{Q}|b(y)-b_{3Q}|^{\d r'}
dy\right)^{\frac{1}{\d r'}} \left( \frac{1}{|Q|}\int_{Q}(Mf(y))^{\d
r} dy\right)^{\frac{1}{\d r}}.
\end{split}
\end{equation*}
By Lemma \ref{lem2.4.} we get
\begin{equation}\label{eqIII}
\begin{split}
{\rm III} \lesssim \|b\|_{*} \left(
\frac{1}{|Q|}\int_{3Q}(Mf(y))^{\e} dy\right)^{\frac{1}{\e}} \leq
\|b\|_{*} M_{\e}(Mf)(x).
\end{split}
\end{equation}
Finally, since $M^2 \thickapprox M_{L \log L}$ (see \cite{CPer}, p.
174 and \cite{graf}, p. 159, for instance), by \eqref{eq003.4},
\eqref{eqI}, \eqref{eqII} and \eqref{eqIII}, we get
\begin{equation}\label{eq000003.1}
M_{\d}(C_b(f))(x)\leq C \|b\|_{*} \left(M_{\e}(Mf)(x)+M^2f(x)\right)
\end{equation}
Since
$$
M_{\e}(Mf)(x)\leq M^2f(x), \quad  \mbox{when} \quad 0<\e<1,
$$
we arrive at \eqref{eqmain2}.
$$\hspace{12.5cm}\square$$

\section{New proof of weak-type $L(\log L)$ estimate for $C_b$}\label{sect4}

In this section we give new proof of the result by Alphonse from
\cite{Alphonse} using Theorem \ref{lem1111111} and Lemma
\ref{lem002.8}. We also show that $b \in \B(\rn)$ is necessary
condition for $C_b$ to be of weak-type $L(\log L)$. The fact that
the operator $C_b$ fails to be of weak type (1,1) follows from
Lemma \ref{pointwise1} and Example \ref{ex4.7}.

\begin{thm}\label{mainthm}
The following assertions are equivalent:

{\rm (i)} There exists a positive constant $C$ such that for each
$\la >0$, inequality \eqref{eq0003.7} holds for all $f\in L(1+\log^+
L)$.

{\rm (ii)} $b\in \B (\rn)$.
\end{thm}

\begin{proof}
{${\rm (i)}\Rightarrow {\rm (ii)}$}. Let $Q_0$ be any fixed cube and
let $f=\chi_{Q_0}$. For any $\la >0$ we have
\begin{align*}
|\{x \in \rn: \,\, C_b(f) (x)  >\la\}| & =|\{x\in \rn: \sup_{x \in
Q}\frac{1}{|Q|}\int_{Q\cap Q_0}|b(x)-b(y)|dy
>\la\}| \\
& \geq |\{x\in Q_0 : \sup_{x \in Q}\frac{1}{|Q|}\int_{Q\cap
Q_0}|b(x)-b(y)|dy
>\la\}|\\
& \geq |\{x\in Q_0 : \frac{1}{|Q_0|}\int_{ Q_0}|b(x)-b(y)|dy
>\la\}|\\
& \geq |\{x\in Q_0 : |b(x)-b_{Q_0}|
>\la\}|,
\end{align*}
since
$$
|b(x)-b_{Q_0}| \leq \frac{1}{|Q_0|}\int_{ Q_0}|b(x)-b(y)|dy.
$$
By assumption the inequality \eqref{eq0003.7} holds for $f$, thus we
have
$$
|\{x\in Q_0 : |b(x)-b_{Q_0}|
>\la\}| \leq C |Q_0|\frac{1}{\la}\left(1+\log^+
\frac{1}{\la}\right).
$$
For $0<\d<1$ we have
\begin{align*}
\int_{Q_0} |b-b_{Q_0}|^{\d}&=\d \int_0^{\i} \la^{\d-1}|\{x\in Q_0
:
|b(x)-b_{Q_0}| >\la\}|d\la\\
&=\d \left\{ \int_0^1 + \int_1^{\i}\right\} \la^{\d-1}|\{x\in Q_0 :
|b(x)-b_{Q_0}| >\la\}|d\la \\
& \leq \d|Q_0|\int_0^1 \la^{\d-1}d\la +
C\d|Q_0|\int_1^{\i}\la^{\d-1} \frac{1}{\la}\left(1+\log^+
\frac{1}{\la}\right)d\la\\
& = |Q_0|+C\d|Q_0|\int_1^{\i}\la^{\d-2} d\la
=\left(1+C\frac{\d}{1-\d}\right)|Q_0|.
\end{align*}
Thus $b\in \B_{\d}(\rn)$. Then by Lemma \ref{lem2.4.} we get that
$b\in \B$.

{${\rm (ii)}\Rightarrow {\rm (i)}$}. By Theorem \ref{lem1111111} and
Lemma \ref{lem002.8}, we have
\begin{align*}
|\{x  \in\rn : C_b(f)(x) & > \la \}| \\
&\leq |\{x  \in\rn : M^2f(x)>
\frac{\la}{C\|b\|_{*}} \}| \\
& \leq C \int_{\rn}\frac{C\|b\|_{*}|f(x)|}{\la}\left(1+\log^+
\left(\frac{C\|b\|_{*}|f(x)|}{\la}\right)\right)dx.
\end{align*}
Since the inequality
\begin{equation}\label{log}
1+\log^+ (ab)\leq (1+\log^+ a)(1+\log^+ b)
\end{equation}
holds for any $a,\,b>0$, we get
\begin{equation*}
\begin{split}
|\{x  \in\rn :\, & C_b(f)(x)> \la \}| \\
& \leq C \|b\|_* (1+\log^+ \|b\|_*)
\int_{\rn}\frac{|f(x)|}{\la}\left(1+\log^+
\left(\frac{|f(x)|}{\la}\right)\right)dx.
\end{split}
\end{equation*}
\end{proof}


\section{Pointwise estimates for the commutator of maximal operator}\label{sect5}

In this section we obtain pointwise estimate of the commutators of
the maximal operator by iterated maximal function.

We shall reduce the study of this commutator to that of $C_b$. The
following lemma is true.
\begin{lem}\label{pointwise1}
Let $b$ be any non-negative locally integrable function on $\rn$.
Then
\begin{equation}\label{eqPointwise}
|[M,b]f(x)|\leq C_b(f)(x)
\end{equation}
for all functions from $L_1^{\loc}(\rn)$.
\end{lem}
\begin{proof}
It is easy to see that for any $f,\, g\in \Lloc $ the following
pointwise estimate holds:
\begin{equation}\label{sublin1}
|Mf(x)-Mg(x)|\leq M(f-g)(x).
\end{equation}
Since $b$ is non-negative, by \eqref{sublin1} we can write
\begin{equation*}
\begin{split}
|[M,b]f(x)|&=|M(bf)(x)-b(x)Mf(x)|=|M(bf)(x)-M(b(x) f)(x)| \\
&\leq M(bf-b(x)f)(x)=M((b-b(x))f)(x)=C_b(f)(x).
\end{split}
\end{equation*}
\end{proof}

Now we can prove estimation \eqref{eq.0001}. At first, let us
formulate the statement.
\begin{lem}\label{lem00004.2}
Let $b$ be any locally integrable function on $\rn$. Then
inequality \eqref{eq.0001} holds for all functions from
$L_1^{\loc}(\rn)$.
\end{lem}
\begin{proof}
Since
$$
|[M,b]f(x)-[M,|b|]f(x)|\leq 2b^-(x) Mf(x)
$$
(see \cite{BasMilRu}, p. 3330, for instance), then
\begin{equation}
|[M,b]f(x)|\leq |[M,|b|]f(x)|+2b^-(x) Mf(x),
\end{equation}
and by Lemma \ref{pointwise1} we have that
\begin{equation*}
|[M,b]f(x)|\leq C_{|b|}f(x)+2b^-(x) Mf(x).
\end{equation*}
Since $||a|-|b||\leq |a-b|$ holds for any $a,b\in \R$, we get
$C_{|b|}f(x)\leq C_b f(x)$ for all $x\in \rn$.
\end{proof}

The following statement follows by Lemma \ref{lem00004.2} and
Theorem \ref{lem1111111}.
\begin{lem}
Let $b$ is in $\B(\rn)$. Then, there exists a positive constant $C$
such that
\begin{equation}
|[M,b]f(x)|\leq C \left(\|b\|_{*}M^2f(x)+b^-(x)Mf(x)\right)
\end{equation}
for all functions from $L_1^{\loc}(\rn)$.
\end{lem}

\begin{cor}\label{cor0000000}
Let $b$ is in $\B(\rn)$ such that $b^-\in L_{\i}(\rn)$. Then, there
exists a positive constant $C$ such that
\begin{equation}
|[M,b]f(x)|\leq C \left(\|b\|_{*}M^2f(x)+\|b^-\|_{\i}Mf(x)\right)
\end{equation}
for all functions from $L_1^{\loc}(\rn)$.
\end{cor}

Now we are in position to prove Theorem \ref{lem0003.3}.

\noindent{\bf \textit {Proof of Theorem \ref{lem0003.3}.}} Statement
follows by Corollary \ref{cor0000000}, since $f \leq Mf$ and
$\|b\|_* \leq \|b^+\|_*+\|b^-\|_*\lesssim
\|b^+\|_*+\|b^-\|_{L_{\i}}$. Note that by triangle inequality
$\|b^+\|_* +\|b^-\|_{\i}\lesssim \|b\|_*+\|b^-\|_{\i}$ holds also.
$$\hspace{12.5cm}\square$$

As mentioned in the introduction, using  estimate
\eqref{eqPointwise3}, we obtain new proof of the sufficient part
of the result by M.~Milman and T.~Schonbek from \cite{MilSchon}.
\begin{cor}
Let $1<p\leq \i$ and let $b$ is in $\B(\rn)$ such that $b^-\in
L_{\i}(\rn)$. Then the commutator $[M,b]$ is bounded in $L_p(\rn)$.
Moreover, there exists a positive constant $C$ such that
\begin{equation}
\|[M,b]f\|_{L_p(\rn)} \leq C \left(\|b^+\|_{*}+\|b^-\|_{\i}\right)
\|f\|_{L_p(\rn)}
\end{equation}
for all functions from $L_p(\rn)$.
\end{cor}

\section{Endpoint estimates for commutators of maximal
function}\label{sect6}

In this section we prove endpoint estimates for the commutators of
the maximal function with $\B$ functions.
\begin{ex}\label{ex4.7}
We show that $[M,b]$ fails to be of weak type $(1,1)$. Consider the
$\B$ function $b(x)=\log |1+x|$ and let $f(x)=\chi_{(0,1)}(x)$. It
is easy to see that for any $x<0$
$$
Mf(x)=\sup_{0<t<1}\frac{t}{t-x}=\frac{1}{1-x}.
$$
On the other hand, for any $x<0$
\begin{equation*}
\begin{split}
M(bf)(x)=\sup_{0<t<1}\frac{\int_0^t
\log|1+y|dy}{t-x}=\sup_{0<t<1}\frac{(1+t)\log
(1+t)-t}{t-x}=\frac{2\log 2-1}{1-x}.
\end{split}
\end{equation*}
Thus
$$
[M,b]f(x)=\frac{2\log 2-1}{1-x}-\frac{\log|1+x|}{1-x}.
$$
There is $\e_0>0$ such that for any $x<-\e_0$
$$
\log|1+x|-(2\log 2-1)>\frac{1}{2}\log|x|.
$$
Therefore, for any $\la >0$,
\begin{equation*}
\begin{split}
\la |\{x\in \R: |[M,b]f(x)|>\la\}|&\geq \la \left|\left\{x<0:
\left|\frac{2\log
2-1}{1-x}-\frac{\log|1+x|}{1-x}\right|>\la\right\}\right| \\
&\geq \la \left|\left\{x<-\e_0:
\frac{1}{2}\frac{\log|x|}{1-x}>\la\right\}\right|\\
&\geq \la \left|\left\{x<-\max\{e,\e_0\}:
\frac{1}{2}\frac{\log|x|}{|x|}>\la\right\}\right|=\\
&= \la (\vp^{-1}(-\max\{e,\e_0\})-\vp^{-1}(2\la)),
\end{split}
\end{equation*}
where $\vp$ is the increasing function $\vp :(-\i, -e)\rightarrow
(0,e^{-1})$, given by $\vp(x)=\log |x|/|x|$. To conclude observe
that the right hand side of the estimate is unbounded as $\la
\rightarrow 0$:
$$
\lim_{\la\rightarrow 0}\la \vp^{-1}(\la)=\lim_{\la\rightarrow \i}\la
\vp(\la)=\i.
$$
\end{ex}

\begin{thm}\label{thm3495195187t}
Let $b$ is in $\B(\rn)$ such that $b^-\in L_{\i}(\rn)$. Then,
there exists a positive constant $C$ such that inequality
\begin{align}
|\{x \in\rn : |[M,b] f(x)|  > \la \}| & \notag\\
& \hspace{-2cm}\leq C C_0\left(1+\log^+
C_0\right)\int_{\rn}\frac{|f(x)|}{\la}\left(1+\log^+
\left(\frac{|f(x)|}{\la}\right)\right)dx, \label{weak}
\end{align}
holds for all $f\in L\left(1+\log^+ L\right)$ and $\la >0$, where
$C_0=\|b^+\|_{*}+\|b^-\|_{\i}$.
\end{thm}

\begin{proof}
By Lemma \ref{lem00004.2}, we have
\begin{equation*}
\begin{split}
|\{x  \in\rn :  |[M,b] & f(x)|> \la \}| \\
&\hspace{-1.5cm}\leq \left|\left\{x  \in\rn :  C_b(f)(x)>
\frac{\la}{2} \right\}\right|+\left|\left\{x  \in\rn :  |2b^-|Mf(x)>
\frac{\la}{2} \right\}\right|\\
&\hspace{-1.5cm}\leq \left|\left\{x  \in\rn :  C_b(f)(x)>
\frac{\la}{2} \right\}\right|+\left|\left\{x  \in\rn :
2\|b^-\|_{\i}Mf(x)> \frac{\la}{2} \right\}\right|.
\end{split}
\end{equation*}
By Theorem \ref{mainthm}, using the inequality \eqref{log}, we
have
\begin{align}
\left|\left\{x  \in\rn :  C_b(f)(x)> \frac{\la}{2} \right\}\right|
& \notag\\
& \hspace{-2.5cm} \leq C C_0\left(1+\log^+
C_0\right)\int_{\rn}\frac{|f(x)|}{\la}\left(1+\log^+
\left(\frac{|f(x)|}{\la}\right)\right)dx. \label{weak1}
\end{align}
On the other hand, since the maximal operator $M$ is a weak type
(1,1), we get
\begin{equation}\label{weak2}
\left|\left\{x  \in\rn : 2\|b^-\|_{\i}\,Mf(x)> \frac{\la}{2}
\right\}\right| \leq C \|b^-\|_{\i}\int_{\rn}\frac{|f(x)|}{\la}dx.
\end{equation}
Combining \eqref{weak1} and \eqref{weak2}, we get \eqref{weak}.

\end{proof}

\begin{rem}
Unfortunately, in Theorem \ref{thm3495195187t} we have only
sufficient part, and we are not able to prove that the condition
$b\in \B (\rn)$ is also necessary for inequality \eqref{weak} to
hold.
\end{rem}


\vspace{1cm}

%
%
%
%
%
%

\

Mujdat Agcayazi \\
Department of Mathematics, Faculty of Science and Arts, Kirikkale
University, 71450 Yahsihan, Kirikkale, Turkey \\
E-mail: mujdatagcayazi@yahoo.com

\

Amiran Gogatishvili\\
Institute of Mathematics of the Academy of Sciences of the Czech
Republic, \'Zitna~25,  115 67 Prague 1, Czech Republic \\
E-mail: gogatish@math.cas.cz

\

Kerim Koca\\
Department of Mathematics, Faculty of Science and Arts, Kirikkale
University, 71450 Yahsihan, Kirikkale, Turkey\\
E-mail: kerimkoca@gmail.com

\

Rza Mustafayev\\
Department of Mathematics, Faculty of Science and Arts, Kirikkale
University, 71450 Yahsihan, Kirikkale, Turkey \\
E-mail: rzamustafayev@gmail.com
%

\end{document}